\documentclass[10pt]{amsart}

\usepackage{hyperref}
\hypersetup{colorlinks=true, urlcolor=blue, citecolor=blue, linkcolor=blue}
\usepackage{enumerate}
\usepackage{amsmath,amsthm, amssymb}
\usepackage[all,cmtip]{xy}

\usepackage{graphicx,mathrsfs,tikz,fancyhdr}

\usepackage{thmtools}
\usepackage{nameref}
\usepackage[backend=bibtex]{biblatex}

\bibliography{Masseybib.bib}

\DeclareMathOperator{\cosupp}{cosupp}
\DeclareMathOperator{\Id}{id}
\DeclareMathOperator{\Gr}{Gr}
\DeclareMathOperator{\var}{var}

\newcommand{\U}{{\mathcal U}}
\newcommand{\0}{{\mathbf 0}}
\newcommand{\C}{{\mathbb C}}

\newcommand{\Z}{{\mathbb Z}}

\newcommand{\Q}{{\mathbb Q}}

\newcommand{\W}{{\mathcal W}}

\newcommand{\HH}{{\mathbb H}}
\newcommand{\cO}{{\mathcal O}}

\newcommand{\hyp}{{\mathbb H}}
\newcommand{\supp}{\operatorname{supp}}

\newcommand{\im}{\mathop{\rm im}\nolimits}

\newcommand{\Idot}{\mathbf {IC}^\bullet}

\newcommand{\Pdot}{\mathbf P^\bullet}

\newcommand{\Mdot}{\mathbf M^\bullet}
\newcommand{\Ndot}{\mathbf  N^\bullet}

\newcommand{\wt}{\widetilde}

\newtheorem{defn0}{Definition}[section]
\newtheorem{prop0}[defn0]{Proposition}
\newtheorem{conj0}[defn0]{Conjecture}
\newtheorem{thm0}[defn0]{Theorem}
\newtheorem{lem0}[defn0]{Lemma}
\newtheorem{corollary0}[defn0]{Corollary}
\newtheorem{example0}[defn0]{Example}
\newtheorem{remark0}[defn0]{Remark}
\newtheorem{question0}[defn0]{Question}

\newenvironment{defn}{\begin{defn0}}{\end{defn0}}
\newenvironment{prop}{\begin{prop0}}{\end{prop0}}
\newenvironment{conj}{\begin{conj0}}{\end{conj0}}
\newenvironment{thm}{\begin{thm0}}{\end{thm0}}
\newenvironment{lem}{\begin{lem0}}{\end{lem0}}
\newenvironment{cor}{\begin{corollary0}}{\end{corollary0}}
\newenvironment{exm}{\begin{example0}\rm}{\end{example0}}
\newenvironment{rem}{\begin{remark0}\rm}{\end{remark0}}
\newenvironment{ques}{\begin{question0}\rm}{\end{question0}}

\newcommand{\defref}[1]{Definition~\ref{#1}}
\newcommand{\propref}[1]{Proposition~\ref{#1}}
\newcommand{\thmref}[1]{Theorem~\ref{#1}}
\newcommand{\lemref}[1]{Lemma~\ref{#1}}
\newcommand{\corref}[1]{Corollary~\ref{#1}}
\newcommand{\exref}[1]{Example~\ref{#1}}
\newcommand{\secref}[1]{Section~\ref{#1}}
\newcommand{\subsecref}[1]{Subsection~\ref{#1}}

\begin{document}
\thispagestyle{plain}

\title{The Weight Filtration on the Constant Sheaf on a Parameterized Space}
\author{Brian Hepler}

\begin{abstract}
On an $n$-dimensional locally reduced complex analytic space $X$ on which the shifted constant sheaf $\Q_X^\bullet[n]$ is perverse, it is well-known that, locally, $\Q_X^\bullet[n]$ underlies a mixed Hodge module of weight $\leq n$ on $X$, with weight $n$ graded piece isomorphic to the intersection cohomology complex $\Idot_X$ with constant $\Q$ coefficients. In this paper, we identify the weight $n-1$ graded piece $\Gr_{n-1}^W \Q_X^\bullet[n]$ in the case where $X$ is a ``parameterized space", using the comparison complex, a perverse sheaf naturally defined on any space for which the shifted constant sheaf $\Q_X^\bullet[n]$ is perverse. 

In the case where $X$ is a parameterized surface, we can completely determine the remaining terms in the weight filtration on $\Q_X^\bullet[2]$, where we also show that the weight filtration is a local topological invariant of $X$. These examples arise naturally as affine toric surfaces in $\C^3$, images of finitely-determined maps from $\C^2$ to $\C^3$, as well as in a well-known conjecture of L\^{e} D\~{u}ng Tr\'{a}ng regarding the equisingularity of parameterized surfaces in $\C^3$.
\end{abstract}

\maketitle 
\section{Introduction}
Mixed Hodge modules (MHM) are at the intersection of many modern branches of algebraic geometry, representation theory, and mathematical physics. As a vast generalization of classical Hodge theory on the cohomology of compact K\"{a}hler varieties, MHM are built out of pairs of a filtered (regular, holonomic) $\mathcal{D}$-module $(\mathcal{M},F^\bullet \mathcal{M})$ and a $\Q$-perverse sheaf $\mathbf{K}^\bullet$ with weight filtration $W_\bullet$ subject to several compatibility conditions (e.g., $DR(\mathcal{M}) \cong \mathbf{K}^\bullet \otimes_\Q \C$ under the Riemann-Hilbert correspondence). 

One of the simplest examples of a mixed Hodge module is the \textbf{constant mixed Hodge module} $\Q_X^H$ on a non-singular complex algebraic variety (or complex analytic manifold) of pure dimension $n$. In this case, the underlying filtered left $\mathcal{D}_X$-module is just $\cO_X$ together with decreasing good filtration $F^p$ satisfying $\Gr_p^F \cO_X = 0$ if $p \neq 0$, and $\Gr_0^F \cO_X = \cO_X$. The associated perverse sheaf is $\Q_X^\bullet[n]$, with weight filtration $W_k$ satisfying $\Gr_k^W \Q_X^\bullet[n] = 0$ if $k \neq n$, and $\Gr_n^W \Q_X^\bullet[n] = \Q_X^\bullet[n]$. For \textbf{singular} complex analytic varieties, however, the constant (mixed) Hodge module is, in general, significantly harder to understand. 

Throughout this paper, we will work in the local complex analytic case, and we will only be concerned with understanding the weight filtration on the associated perverse sheaf $\Q_X^\bullet[n]$ of the constant mixed Hodge module $\Q_X^H$. In particular, for a class of complex analytic spaces called \textbf{parameterized spaces} (\defref{def:paramspace}), we determine the graded component $\Gr_{n-1}^W \Q_X^\bullet[n]$ (\thmref{thm:graded}). We give special attention to parameterized surfaces, where we completely determine the weight filtration on $\Q_X^\bullet[2]$ (\thmref{thm:weightzeropart},\thmref{thm:semispacevspace}) using down-to-earth geometric calculations, and show that this weight filtration is a local topological invariant of $X$ near a given point. (\thmref{thm:topinv}). In \subsecref{subsec:surfacesvancycle}, we are able to completely determine the monodromy weight filtration on the unipotent vanishing cycles for any parameterized surface $V(f)$ in $\C^3$.

\smallskip

 We would like to express our thanks to J\"{o}rg Sch\"{u}rmann for suggesting a simplified version of our original proofs of \propref{prop:ndotvhs} and  \thmref{thm:graded}, as well as many helpful discussions on the connection with the vanishing cycles, resulting in \secref{sec:vancycle}. \lemref{lem:getses} is due to the author and David Massey.

\bigskip

\section{Basic Notions}\label{sec:basics}

Let $\mathcal{W}$ be an open neighborhood of the origin in $\C^N$, let $X \subseteq \W$ be a (reduced) complex analytic space containing $\0$ of pure dimension $n$,  on which the (shifted) constant sheaf $\Q_X^\bullet[n]$ is perverse (e.g., if $X$ is a local complete intersection).

There is then a surjection of perverse sheaves $\Q_X^\bullet[n] \to \Idot_X \to 0$, where $\Idot_X$ is the intersection cohomology complex on $X$ with constant $\Q$ coefficients. Since the category of perverse sheaves is Abelian, we obtain a short exact sequence
\begin{equation}\label{eqn:fundses}
0 \to \Ndot_{X} \to \Q_{X}^\bullet[n] \to \Idot_{X} \to 0.
\end{equation}
The perverse sheaf $\Ndot_{X}$ is called the \textbf{comparison complex} on $X$, and was first defined by the author and David Massey in \cite{hepmasparam} and subsequently studied in several papers by the author \cite{hepdefhyper}, \cite{qhomcriterion} and Massey \cite{comparison}. 

\bigskip

By shrinking $\W$ if necessary, the perverse sheaf $\Q_X^\bullet[n]$ underlies a graded-polarizable mixed Hodge module (Prop 2.19, Prop 2.20, \cite{mixedhodgemod}) of weight $\leq n$. Moreover, by Morihiko Saito's theory of (graded polarizable) mixed Hodge modules in the local complex analytic context, the perverse cohomology objects of the usual sheaf functors naturally lift to cohomology functors in the context of (graded polarizable) mixed Hodge modules (but not on their derived category level as in the algebraic context as in Section 4 of \cite{mixedhodgemod}). Moverover, by (4.5.9) \cite{mixedhodgemod}, the quotient morphism $\Q_X^\bullet[n] \to \Idot_X$ induces an isomorphism 
$$
\Gr_n^W \Q_X^\bullet[n] \overset{\thicksim}{\to} \Idot_X;
$$
consequently, the short exact sequence (\ref{eqn:fundses}) identifies the comparison complex $\Ndot_X$ with $W_{n-1}\Q_{X}^\bullet[n]$. This then endows $\Ndot_{X}$ with the structure of a mixed Hodge module of weight $\leq n-1$ with weight filtration $W_k \Ndot_{X} = W_k \Q_{X}^\bullet[n]$ for $k \leq n-1$. 

\bigskip

Before we state can state and prove our main results, we first recall a theorem of Borho and MacPherson \cite{BorhoMac} giving us several equivalent characterizations of rational homology manifolds (or, \textbf{$\Q$-homology manifolds}):

\begin{thm}([B-M])\label{thm:ICman}
The following are equivalent:
\begin{enumerate}

\item $X$ is a $\Q$-homology manifold, i.e., for all $p \in X$, 
$$
H^k(X,X\backslash \{p\};\Q) \cong 
\begin{cases}
\Q, & \text{ if $k = 2n$,} \\ 
0, & \text{ if $k \neq 2n$.}
\end{cases}
$$

\item  The natural morphism $\Q_X^\bullet[n] \to \Idot_X$ is an isomorphism.

\smallskip

\item $\mathcal{D}\left ( \Q_X^\bullet[n] \right ) \cong \Q_X^\bullet[n]$, where $\mathcal{D}$ is the Verdier duality functor.

\end{enumerate}
\end{thm}

Let $\pi :(\wt X,S) \to (X,\0)$ be the normalization of $X$, where $S := \pi^{-1}(\0)$. The normalization map is a \emph{small map} in the sense of Goresky and MacPherson \cite{inthom2}, and so there is an isomorphism $\pi_*\Idot_{\wt X} \cong \Idot_{X}$, where $\Idot_{\wt X}$ is intersection cohomology on $\wt X$ with constant $\Q$ coefficients. Thus, when the normalization is a $\Q$-homology manifold, $\Idot_{X}\cong \pi_*\Q_{\wt X}^\bullet[n]$. In this case, by taking the long exact sequence in stalk cohomology of (\ref{eqn:fundses}), we then find that $\Ndot_{X}$ has cohomology concentrated in degree $-n+1$, and in that degree, we have $\dim H^{-n+1}(\Ndot_{X})_p = |\pi^{-1}(p)|-1$.  

From this, it follows that 
$$
D_X := \supp \Ndot_{X} = \overline{\{p \in X \, | \, |\pi^{-1}(p)| > 1\}}
$$
is a purely $(n-1)$-dimensional set (it is the support of a perverse sheaf concentrated in degree $-n+1$), and $D_X \subseteq \Sigma X$. 

\bigskip

\begin{rem}\label{rem:warning}
\textbf{Throughout this paper, we will assume the normalization of $X$ is a rational homology manifold}; additionally, we will assume that $D_X = \Sigma X$, so that $\Sigma X$ will always be purely $(n-1)$-dimensional. 

 We do, however, drop this assumption in the proof of the topological invariance of the weight filtration for parameterized surfaces in \thmref{thm:topinv}. Although it is most instructive to think of $D_X$ as being the entire singular locus of $X$ for most applications, $\Sigma X$ is unfortunately not an invariant of the local topological type of $X$, whereas $D_X$ is invariant. 
 
 For example, the parameterized surface $X = V(y^2-x^3)$ in $\C^3$ is homeomorphic to $\C^2$, but $X$ is singular and $\C^2$ is clearly not. However, we do have $D_X = \emptyset$ in the first case; this distinction is crucial in the proof of \thmref{thm:topinv}.
\end{rem}

We have called such spaces $X$ with $\Q$-homology manifold normalizations \textbf{parameterized spaces} in \cite{hepmasparam},\cite{hepdefhyper}, and \cite{qhomcriterion}.

\begin{defn}\label{def:paramspace}
A reduced, purely $n$-dimensional space on which $\Q_X^\bullet[n]$ is perverse and for which the normalization $\wt X$ of $X$ is a $\Q$-homology manifold is called a \textbf{parameterized space}.
\end{defn}

\medskip

We will also use the following result throughout this paper, in which the vanishing of the cohomology sheaves of the comparison complex $\Ndot_{X}$ places strong constraints on the topology of the normalization $\wt X$. 

\begin{thm}[H., \cite{qhomcriterion}]\label{thm:qhomcriterion}
$X$ is a parameterized space if and only if $\Ndot_{X}$ has cohomology sheaves concentrated in degree $-n+1$; i.e., for all $p \in X$, $H^k(\Ndot_{X})_p$ is non-zero only possibly when $k= -n+1$.
\end{thm}

 Letting $\Sigma X$ denote the singular locus of $X$, and let $i : \Sigma X \hookrightarrow X$. We can then find a smooth, Zariski open dense subset $\U \subseteq \Sigma X$ over which the normalization map restricts to a covering projection $\hat{\pi} : \pi^{-1}(\U) \to \U \subseteq \Sigma X$ (see Section 6.2, \cite{inthom2}). Let $l : \U \hookrightarrow \Sigma X$ and $m : \Sigma X \backslash \U \hookrightarrow \Sigma X$ denote the respective open and closed inclusion maps. Let $\hat{m} := i \circ m$, $\hat{l} : = i \circ l$. Note that $\dim_\0 \Sigma X \backslash \U \leq n-2$, as it is the complement of a Zariski open set (we will need this later in \propref{prop:ndotsupp}).
 
 \begin{exm}
Consider the Whitney umbrella $V(f) \subseteq \C^3$ with $f(x,y,z) = y^2-x^3-zx^2$. Then, the normalization of $V(f)$ is smooth, and given by the map $\pi(u,t) = (u^2-t,u(u^2-t),t)$. 

The critical locus of $f$ is $\Sigma f = V(x,y)$, and it is easy to see that over $\Sigma f \backslash \{\0\}$, $\pi$ is a 2-to-1 covering map; thus, we set $\U = \Sigma f \backslash \{\0\}$.
 \end{exm}
 
 \begin{exm}\label{exm:surfacecase}
Suppose $V(f) \subseteq \C^3$ is a parameterized surface with $\dim_\0 \Sigma f =1$. Then, it is easy to see that $\U = \Sigma f \backslash \{\0\}$; this follows from the fact that ${\Idot_{V(f)}}_{|_{\Sigma f}}$ is is constructible with respect to the Whitney stratification $\{\Sigma f \backslash \{\0\},\{\0\}\}$ of $\Sigma f$, along with the description of the stalk cohomology of $\Idot_{V(f)}$ given by the isomorphism $\Idot_{V(f)} \cong \pi_*\Q_{\wt {V(f)}}^\bullet[2]$. 

We will examine this setting in more detail in \secref{sec:surfaces}.
 \end{exm}

 \bigskip


\section{Main Result}

In this section, we first prove a general result, \lemref{lem:getses}, about perverse sheaves that will allow us to construct the short exact sequence mentioned in \thmref{thm:graded}, and that $\Ndot_{X}$ satisfies the hypotheses of this lemma. Then, we examine the weight filtration on $\Idot_{\Sigma X}(\hat{l}^*\Ndot_{X})$ and show that it underlies a polarizable Hodge module of weight $n-1$ in \propref{prop:ndotvhs}. With all this, we can state and prove \thmref{thm:graded} and \corref{cor:actualthm}. 

\bigskip

Recall the category of perverse sheaves $Perv(X)$ is the Abelian subcategory of the bounded derived category of $\C$-constructible sheaves $D_c^b(X)$ given by the heart of the \textbf{perverse $t$-structure}, $Perv(X) = {}^p D^{\leq 0}(X) \cap {}^p D^{\geq 0}(X)$. Here, 
\begin{itemize}

\item $\Pdot \in {}^p D^{\leq 0}(X)$ if $\Pdot$ satisfies the \textbf{support condition}: for all $k \in \Z$, 
$$
\dim_\C \supp H^k(\Pdot) \leq -k.
$$

\item $\Pdot \in {}^p D^{\geq 0}(X)$ if $\mathcal{D}\Pdot$ satisfies the support condition, where again $\mathcal{D}$ denotes the Verdier duality functor. More precisely, for all $k \in \Z$,
$$
\dim_\C \overline{ \{p \in X \, | \, H^k(i_p^! \Pdot) \neq 0 \}} \leq k,
$$
where $i_p : \{p\} \hookrightarrow X$. This is known as the \textbf{cosupport condition}.

\end{itemize}. 

The following lemma is necessary to construct the short exact sequence appearing in \thmref{thm:graded}, although it is a much more general result about arbitrary perverse sheaves on analytic spaces.

\begin{lem}\label{lem:getses}
Suppose $X$ is a complex analytic space, $\Pdot$ a perverse sheaf on $X$, $l: \U \hookrightarrow X$ a Zariski open subset and $m : Z = X \backslash \U \hookrightarrow X$ its closed analytic complement. Then, if $m^*[-1]\Pdot \in {}^p D^{\leq 0}(Z)$, there is a short exact sequence
$$
0 \to m_*{}^p H^0(m^!\Pdot) \to \Pdot \to \Idot_X(l^*\Pdot) \to 0
$$
of perverse sheaves on $X$, where $\Idot_X(l^*\Pdot) := \im {}^p H^0(l_!l^*\Pdot \to l_*l^*\Pdot)$ denotes the intermediate extension of $l^*\Pdot$ to all of $X$. 
\end{lem}

\begin{proof}
The natural morphism ${}^p H^0(l_!l^*\Pdot) \to {}^p H^0(l_*l^*\Pdot)$ factors as 
$$
{}^p H^0(l_!l^*\Pdot) \overset{\alpha}{\to} \Pdot \overset{\beta}{\to} {}^p H^0(l_*l^*\Pdot).
$$
 From the other natural distinguished triangle associated to this pair of subsets, 
$$
l_!l^*\Pdot \to \Pdot \to m_*m^*\Pdot \overset{+1}{\to},
$$
we see that surjectivity of $\alpha$ follows from the vanishing of 
$$
{}^p H^0(m_*m^*\Pdot) \cong m_*{}^p H^0(m^*\Pdot).
$$
By assumption, $m^*[-1]\Pdot \in {}^p D^{\leq 0}(Z)$, so that ${}^p H^k(m^*[-1]\Pdot) = 0$ for all $k >0$. Thus,
$$
{}^p H^0(m^*\Pdot) \cong {}^p H^1(m^*[-1]\Pdot) = 0;
$$
hence, $\alpha$ is surjective, and we have $\im \beta = \im (\beta \circ \alpha) \cong \Idot_X(l^*\Pdot)$. We then obtain the isomorphism $\Idot_X(l^*\Pdot) \cong \im \{ \Pdot \to {}^p H^0(l_*l^*\Pdot) \}$. 

Finally, the result follows from the long exact sequence in perverse cohomology associated to the distinguished triangle
$$
m_*m^!\Pdot \to \Pdot \to l_*l^*\Pdot \overset{+1}{\to},
$$
since $m_*m^!\Pdot \in {}^p D^{\geq 0}(X)$ and $l_*l^*\Pdot \in {}^p D^{\geq 0}(X)$,  (see, e.g.,  Proposition 10.3.3 of \cite{kashsch}, or Theorem 5.2.4 of \cite{dimcasheaves}). 
\end{proof}

\bigskip

From the introduction, let $\Sigma X$ denote the singular locus of $X$, and let $i : \Sigma X \hookrightarrow X$. We can then find a smooth, Zariski open dense subset $\U \subseteq \Sigma X$ over which the normalization map restricts to a covering projection $\hat{\pi} : \pi^{-1}(\U) \to \U \subseteq \Sigma X$ (see Section 6.2, \cite{inthom2}). Let $l : \U \hookrightarrow \Sigma X$ and $m : \Sigma X \backslash \U \hookrightarrow \Sigma X$ denote the respective open and closed inclusion maps. Let $\hat{m} := i \circ m$, $\hat{l} : = i \circ l$. Note that $\dim_\0 \Sigma X \backslash \U \leq n-2$, as it is the complement of a Zariski open set.


\begin{prop}\label{prop:ndotsupp}
If $X$ is a parameterized space, then $\hat{m}^*[-1]\Ndot_{X} \in {}^p D^{\leq 0}(\Sigma X \backslash \U)$.
\end{prop}
\begin{proof}

We wish to show that for all $k \in \Z$, 
$$
\dim_\C \supp H^k(\hat{m}^*[-1]\Ndot_{X}) \leq -k.
$$
However, $\supp H^k(\hat{m}^*[-1]\Ndot_{X})$ is non-empty only for $k-1 = -n+1$, i.e., when $k = -n+2$ (by \thmref{thm:qhomcriterion}). In this degree, the support is equal to $\Sigma X \backslash \U$. Since this set is the complement of a Zariski open dense subset of $\Sigma X$, 
$$
\dim_\C \supp H^{-n+2}(\hat{m}^*\Ndot_{X}) \leq n-2,
$$
as desired.
\end{proof}

\begin{rem}
For surfaces $X$ with curve singularities, $\hat{m}^*[-1]\Ndot_X \in {}^p D^{\leq 0}(\Sigma X \backslash \U)$ if and only if $X$ is parameterized (see \secref{sec:surfaces}).

In general, $\hat{m}^*[-1]\Ndot_{X} \in {}^p D^{\leq 0}(\Sigma X \backslash \U)$ places strict constraints on the possible cohomology groups of the \textbf{real link} of $X$ at different points $p \in \Sigma X$, denoted $K_{X,p}$, i.e., the intersection of $X$ with a sphere of sufficiently small radius at $p$. 
\end{rem}

\bigskip


\begin{rem}\label{rem:saitodurfee}
Generically along an irreducible component $C$ of $\Sigma X$, $\Ndot_{X}$ is isomorphic to a local system $\hat{l}^*({\Ndot_{X}}_{|_C})$ in degree $-n+1$, and in that degree, we have
$$
H^{-n+1}(\Ndot_{X})_p \cong \widetilde \HH^{-n}(K_{X,p};\Idot_{X}),
$$
where $\widetilde \HH$ denotes reduced hypercohomology, This description follows immediately from short exact sequence (\ref{eqn:fundses}). Since $\Idot_{X} \cong \pi_*\Q_{\wt X}^\bullet[n]$, this reduced hypercohomology is actually just 
\begin{align*}
\widetilde \HH^{-n}(K_{X,p};\Idot_{X}) \cong \widetilde H^0(K_{\wt X,\pi^{-1}(p)};\Q),
\end{align*}
where 
$$
K_{\wt X,\pi^{-1}(p)} = \bigcup_{q \in \pi^{-1}(p)} K_{\wt X,q}.
$$
Since $\wt X$ is normal (and thus locally irreducible) it is clear that one has $H^0(K_{\wt X,q};\Q) \cong \Q$ for all $q \in \wt X$. After noting that $H^{-n}(\Idot_{X})_p = IH^0(K_{X,p})$ (that is, intersection cohomology of $K_{X,p}$ with topological indexing), $H^{-n}(\Idot_{X})_p$ has a pure Hodge structure of weight $0$ (see, e.g., A. Durfee and M. Saito \cite{durfee1990}). 
\end{rem}
\medskip

\begin{prop}\label{prop:ndotvhs}
Let $C$ be an irreducible component of $\Sigma X$ at $\0$. Then, $\hat{l}^*({\Ndot_{X}}_{|_C})$ underlies a polarizable variation of Hodge structure of weight $0$. 

Consequently, $\Idot_{\Sigma X}(\hat{l}^*\Ndot_{X})$ underlies a polarizable Hodge module of weight $n-1$ on $\Sigma X$.
\end{prop}
\begin{proof}
Since $\hat{l}^*\Ndot_{X}$ underlies a mixed Hodge module whose underlying perverse sheaf is a local system (up to a shift) on the complex manifold $\U$, this local system underlies an admissable graded polarizable variation of mixed Hodge structures on $\U$ by Theorem 3.27 \cite{mixedhodgemod}. 

To show that this mixed Hodge structure is pure of weight zero, we can check on stalks at points $p \in \U$. Let $i_p : \{p\} \hookrightarrow \U$; then, the stalk cohomology $H^k(-)_p$ agrees with perverse cohomology ${}^p H^k(i_p^*)$. So, applying $H^k(i_p^*)$ on the level of mixed Hodge modules to the short exact sequence (\ref{eqn:fundses}), we get by Proposition 2.19, Proposition 2.20, and Theorem 3.9 of \cite{mixedhodgemod} a short exact sequence in the category of graded polarizable mixed Hodge structures, whose underlying sequence of vector spaces is
\begin{equation}\label{eqn:vectses}
0 \to \Q_{\{p\}} \to H^{-n}(\Idot_{X})_p \to H^{-n+1}(\Ndot_{X})_p \to 0.
\end{equation}
However, $\pi :\wt X \to X$ is a finite map, and therefore exact for the perverse $t$-structure (and mixed Hodge modules), with 
$$
H^{-n}(\Idot_{X})_p \cong H^{-n}(\pi_*\Q_{X}^\bullet[n])_p \cong \bigoplus_{y \in \pi^{-1}(p)} \Q_{\{y\}}.
$$
Since this stalk is pure of weight zero, the surjection in (\ref{eqn:vectses}) implies $H^{-n}(\Ndot_{X})_p$ is also pure of weight zero.

\end{proof}

\bigskip

From the introduction, we have the inclusions $\hat{m} : \Sigma X \backslash \U \hookrightarrow X$ and $\hat{l} : \U \hookrightarrow X$, which give the distinguished triangle
$$
m_*m^!i^*\Ndot_{X} \to i^*\Ndot_{X} \to l_*\hat{l}^*\Ndot_{X} \overset{+1}{\to}.
$$
By \lemref{lem:getses}, \propref{prop:ndotsupp}, and \propref{prop:ndotvhs} we now have a short exact sequence of perverse sheaves coming from a short exact sequence of mixed Hodge modules (Corollary 2.20 \cite{mixedhodgemod})
\begin{equation}\label{eqn:grses}
0 \to m_*{}^p H^0(m^!i^*\Ndot_{X}) \to i^*\Ndot_{X} \to \Idot_{\Sigma X}(\hat{l}^*\Ndot_{X}) \to 0
\end{equation}
where $i^*\Ndot_{X}$ has weight $\leq n-1$ (recall $\Ndot_{X}$ has weight $\leq n-1$, and $i^*$ does not increase weights \cite{peterssteenbrink} pg. 340), and $\Idot_{\Sigma X}(\hat{l}^*\Ndot_{X})$ has weight $n-1$. Since a short exact sequence of mixed Hodge modules is strictly compatible with the weight filtration, and the functor $\Gr_{n-1}^W$ is exact on the Abelian category of polarizable mixed Hodge modules, we have the short exact sequence of mixed Hodge modules and their underlying perverse sheaves
$$
0 \to \Gr_{n-1}^W m_*{}^p H^0(m^!i^*\Ndot_{X}) \to \Gr_{n-1}^W i^*\Ndot_{X} \to \Idot_{\Sigma X}(\hat{l}^*\Ndot_{X}) \to 0.
$$

\bigskip

We can now state and prove our main result.





\begin{thm}\label{thm:graded}
Suppose $X$ is a parameterized space. Then, there is an isomorphism $\Gr_{n-1}^W i^*\Ndot_{X} \cong \Idot_{\Sigma X}(\hat{l}^*\Ndot_{X})$, so that the short exact sequence of perverse sheaves on $X$
\begin{equation}\label{eqn:gradedses}
0 \to m_*{}^p H^0(m^!i^*\Ndot_{X}) \to i^*\Ndot_{X} \to \Idot_{\Sigma X}(\hat{l}^*\Ndot_{X}) \to 0
\end{equation}
identifies $W_{n-2}i^*\Ndot_{X} \cong m_*{}^p H^0(m^!i^*\Ndot_{X})$. 
\end{thm}

\begin{proof}
Since $\Gr_{n-1}^W i^*\Ndot_{X}$ underlies a pure Hodge module, it is by definition semi-simple as a perverse sheaf, i.e., a direct sum of simple intersection cohomology sheaves with irreducible support. Hence, we can write $\Gr_{n-1}^W i^*\Ndot_{X}$ as direct sum of a semi-simple perverse sheaf $\Mdot$ with support in $\Sigma X \backslash \U$ and a semi-simple perverse sheaf whose summands are all not supported on $\Sigma X \backslash \U$. This second semi-simple perverse sheaf has to be $\Idot_{\Sigma X}(\hat{l}^*\Ndot_{X})$, by pulling back the short exact sequence (\ref{eqn:grses}) by $\hat{l}^*$. 

\medskip

Finally, we claim $\Mdot = 0$. Since $\Mdot$ is a direct summand of $\Gr_{n-1}^W i^*\Ndot_{X}$, we have a surjection of perverse sheaves
$$
i^*\Ndot_{X} \to \Gr_{n-1}^W i^*\Ndot_{X} \to \Mdot.
$$
But ${}^p H^0(m^*)$ is right exact for the perverse t-structure (since $m^*$ is a closed inclusion), so we also get a surjection
$$
0 = {}^p H^0(\hat{m}^*\Ndot_{X}) \to {}^p H^0(m^*\Mdot) = \Mdot \to 0,
$$
where the last equality follows from the fact that $\Mdot$ is supported on $\Sigma X \backslash \U$.

\end{proof}

\smallskip

\begin{cor}\label{cor:actualthm}
There are isomorphisms 
$$
\Gr_{n-1}^W \Q_{X}^\bullet[n] \cong \Gr_{n-1}^W \Ndot_{X} \cong i_*\Idot_{\Sigma X}(\hat{l}^*\Ndot_{X}),
$$
and 
$$
W_{n-2}\Q_{X}^\bullet[n] \cong W_{n-2}\Ndot_{X} \cong \hat{m}_*{}^p H^0(m^!i^*\Ndot_{X}).
$$
\end{cor}

As mentioned in the introduction, this trivially follows from the fact that $i_*$ preserves weights (\cite{peterssteenbrink}, pg. 339), is exact for the perverse $t$-structure, and from the fact that $i_*i^*\Ndot_{X} \cong \Ndot_{X}$, since the support of $\Ndot_{X}$ is contained in $\Sigma X$. 

\begin{exm}
If $X$ is a simple normal crossing divisor in a smooth projective variety over $\C$, then the normalization of $X$ is smooth and results of \corref{cor:actualthm} hold. When $X$ is purely 2-dimensional, we can determine the entire weight filtration on $X$ by the results of \secref{sec:surfaces} below.
\end{exm}

\bigskip

At first glance, the formula for $W_{n-2}i^*\Ndot_X$ appears quite abstruse. We now give a much more geometric interpretation of this perverse sheaf. 

\begin{thm}\label{thm:geomweight}
Let $g$ be a complex analytic function on $\Sigma X$ such that $V(g)$ contains $\Sigma X \backslash \U$, but does not contain any irreducible component of $\Sigma X$. Then, 
$$
W_{n-2}i^*\Ndot_X \cong {m'}_*\ker \{ \phi_g[-1]i^*\Ndot_X \overset{\var}{\longrightarrow} \psi_g[-1]i^*\Ndot_X \},
$$
where the kernel is taken in the category of perverse sheaves on $\Sigma X$, $\var$ is the variation morphism, and $m' : V(g) \hookrightarrow \Sigma X$ is the closed inclusion.
\end{thm}

\begin{proof}
We first note that such a function $g$ exists locally by the prime avoidance lemma. Then, $\Sigma X \backslash V(g) \subseteq \U$, and we have as perverse sheaves
$$
\Idot_{\Sigma X}({i^*\Ndot_X}_{|_{\Sigma X \backslash V(g)}}) \cong \Idot_{\Sigma X}(\hat{l}^*\Ndot_X),
$$
since the normalization is still a covering projection away from $V(g)$ in $\Sigma X$. One notes then that the proofs of \propref{prop:ndotsupp}, \propref{prop:ndotvhs}, and \thmref{thm:graded} remain unchanged with these new choices of complementary subspaces $V(g) \overset{m'}{\hookrightarrow} \Sigma X$ and $\Sigma X \backslash V(g) \overset{l'}{\hookrightarrow} \Sigma X$, so that 
$$
\Gr_{n-1}^W i^*\Ndot_X \cong \Idot_{\Sigma X}({i^*\Ndot_X}_{|_{\Sigma X \backslash V(g)}})
$$
and 
$$
W_{n-2} i^*\Ndot_X \cong m'_*{}^p H^0({m'}^!i^*\Ndot_X).
$$


The claim then follows by taking the long exact sequence in perverse cohomology of the variation distinguished triangle 
$$
\phi_g[-1]i^*\Ndot_X \overset{\var}{\longrightarrow} \psi_g[-1]i^*\Ndot_X \to {m'}^![1]i^*\Ndot_X \overset{+1}{\longrightarrow},
$$
yielding
$$
0 \to {}^p H^0({m'}^!i^*\Ndot_X) \to \phi_g[-1]i^*\Ndot_X \overset{\var}{\longrightarrow} \psi_g[-1]i^*\Ndot_X \to {}^p H^1({m'}^!i^*\Ndot_X) \to 0.
$$
\end{proof}

\bigskip

\section{Connection with the Vanishing Cycles}\label{sec:vancycle}

In \cite{comparison}, Massey shows that, for an arbitrary (reduced) hypersurface $V(f)$ in some open neighborhood $\W$ of the origin in $\C^{n+1}$, one has a isomorphism of perverse sheaves $\Ndot_{V(f)} \cong \ker \{\Id-\wt T_f\}$, where $\wt T_f$ is the Milnor monodromy action on the vanishing cycles $\phi_f[-1]\Z_\W^\bullet[n+1]$ (this isomorphism holds for $\Q$ coefficients, where one may also obtain this result using the language of mixed Hodge modules). 

However, $\Id -\wt T_f$ is not a morphism of mixed Hodge modules; to remedy this, one instead considers the morphism $N = \frac{1}{2\pi i} \log T_u$, where $T_u$ is the unipotent part of the monodromy operator $\wt T_f$. In this case, $\ker \{\Id - \wt T_f\} \cong \ker N$ as perverse sheaves, and we consider $\ker N$ as a subobject of the unipotent vanishing cycles $\phi_{f,1}[-1]\Q_\W^\bullet[n+1]$. 

On the level of mixed Hodge modules, we have an isomorphism
$$
\Ndot_{V(f)} \cong \ker N (1)
$$
where $(1)$ denotes the Tate twist operation. This description follows from Massey's original proof for perverse sheaves \cite{comparison}, with the following changes. Starting from the two short exact sequences of mixed Hodge modules
\begin{equation}\label{eqn:milnorses}
0 \to j^*[-1]\Q_\W^\bullet[n+1] \to \psi_{f,1}[-1]\Q_\W^\bullet[n+1] \xrightarrow{can} \phi_{f,1}[-1]\Q_\W^\bullet[n+1] \to 0
\end{equation}
and 
$$
0 \to \phi_{f,1}[-1]\Q_\W^\bullet[n+1] \xrightarrow{var} \psi_{f,1}[-1]\Q_\W^\bullet[n+1](-1) \to j^![1]\Q_\W^\bullet[n+1] \to 0,
$$
(note the variation morphism now has a Tate twist of $(-1)$), so that $N = can \circ var$. Then, if $i : \Sigma f \hookrightarrow V(f)$, we obtain the isomorphism 
$$
\phi_{f,1}[-1]\Q_\W^\bullet[n+1](1) \xrightarrow{i_*{}^p H^0(i^! var)} i_*{}^p H^0(i^!\psi_{f,1}[-1]\Q_\W^\bullet[n+1])
$$
 since $(j \circ i)^![1]\Q_\W^\bullet[n+1] \in {}^p D^{\geq 0}(V(f) \backslash \Sigma f)$. This, together with the isomorphisms 
 $$
 \Ndot_{V(f)} \cong i_*{}^p H^0(i^!j^*[-1]\Q_\W^\bullet[n+1]) \cong i_*{}^p H^0(i^!\ker can)
 $$
 obtained by \lemref{lem:getses} and applying $i_*{}^p H^0(i^!)$ to (\ref{eqn:milnorses}) yields the final identification
$$
 \Ndot_{V(f)} \cong \ker N (1).
 $$
 Hence, $W_k \Ndot_{V(f)} \cong W_{k+2}\ker N$ for all $k \leq n-1$. 
 
 The unipotent vanishing cycles $\phi_{f,1}[-1]\Q_\W^\bullet[n+1]$, as a mixed Hodge module, is endowed with the monodromy weight filtration shifted by $n+1$, via the nilpotent operator $N$ (see e.g., \cite{peterssteenbrink} or \cite{modulesdehodge}). Hence, for all $k$,
 $$
 N(W_k \phi_{f,1}[-1]\Q_\W^\bullet[n+1]) \subseteq W_{k-2} \phi_{f,1}[-1]\Q_\W^\bullet[n+1],
 $$
 and there are isomorphisms
 \begin{equation}\label{eqn:hardlefschetz}
 \Gr_{n+1+k}^W \phi_{f,1}[-1]\Q_\W^\bullet[n+1] \xrightarrow{\thicksim} \Gr_{n+1-k}^W \phi_{f,1}[-1]\Q_\W^\bullet[n+1]
 \end{equation}
 for all $k \geq 0$. These isomorphisms are a vast generalization of the Hard Lefschetz Theorem for the cohomology of compact K\"{a}hler varieties.
 
 We will examine this again in \subsecref{subsec:surfacesvancycle}, in the context of parameterized surfaces $V(f)$ in $\C^3$.

 \bigskip

\section{The Surface Case}\label{sec:surfaces}
Suppose $X$ is a parameterized surface; we want to compute $W_0\Q_X^\bullet[2]$ using the isomorphism
$$
W_0 \Q_X^\bullet[2] = W_0\Ndot_X \cong \hat{m}_*{}^p H^0(m^!i^*\Ndot_X).
$$

\subsection{Computing $W_0$}\label{subsec:weightzero}
The main tool we use is the following: if $\dim_\0 \Sigma X = 1$, then $\Sigma X \backslash \U$ is zero dimensional (or empty), and perverse cohomology on a zero-dimensional space is just ordinary cohomology. Recall that $\wt X \overset{\pi}{\to} X$ is the normalization map. 

\begin{thm}\label{thm:weightzeropart}
 Suppose $X$ is a parameterized surface. Then, 
$$
W_0\Q_X^\bullet[2] \cong V_{\{\0\}}^\bullet
$$
is a perverse sheaf concentrated on a single point, i.e., a finite-dimensional $\Q$-vector space, of dimension 
\begin{align*}
\dim V = 1-|\pi^{-1}(\0)| + \sum_C \dim \ker\{\Id -h_C\},
\end{align*}
where $\{C\}$ is the collection of irreducible components of $\Sigma X$ at $\0$, and for each component $C$,  $h_C$ is the (internal) monodromy operator on the local system $H^{-1}(\Ndot_X)_{|_{C\backslash \{\0\}}}$.  Note that $|\pi^{-1}(\0)|$ is, of course, equal to the number of irreducible components of $X$ at $\0$.

\end{thm}

\begin{proof}

First, note that we have $\Sigma X \backslash \U = \{\0\}$, and $\U = \bigcup_C (C \backslash \{\0\})$, where each $C \backslash \{\0\}$ is homeomorphic to a punctured complex disk. Then, we find
\begin{align*}
{}^p H^0(m^!i^*\Ndot_{X}) &\cong H^0(m^!i^*\Ndot_{X}) \cong \HH^0(\Sigma X,\Sigma X \backslash \{\0\};i^*\Ndot_{X}) 
\end{align*}
We can compute this last term from the long exact sequence in relative hypercohomology with coefficients in $\Ndot_{X}$:
\begin{align*}
0 \to \HH^{-1}(\Sigma X,\Sigma X \backslash \{\0\};i^*\Ndot_{X}) &\to H^{-1}(\Ndot_{X})_\0 \to \HH^{-1}(\Sigma X \backslash \{\0\};i^*\Ndot_{X}) \to \\
 \HH^0(\Sigma X,\Sigma X \backslash \{\0\};i^*\Ndot_{X}) &\to H^0(\Ndot_{X})_\0 \to \HH^0(\Sigma X \backslash \{\0\};i^*\Ndot_{X}) \to 0 \\
\end{align*}
The cosupport condition on $i^*\Ndot_{X}$ implies $\HH^{-1}(\Sigma X, \Sigma X \backslash \{\0\};i^*\Ndot_{X}) = 0$. Additionally, since $H^0(\Ndot_{X})$ is only supported on $\{\0\}$, it follows that $\HH^0(\Sigma X \backslash\{\0\};i^*\Ndot_{X}) = 0$ as well. Since the normalization of $X$ is rational homology manifold, $H^0(\Ndot_{X})_\0 = 0$ by \thmref{thm:qhomcriterion}, and $\dim H^{-1}(\Ndot_{X})_\0 = |\pi^{-1}(\0)|-1$. 

\smallskip

Finally, 
\begin{align*}
\HH^{-1}(\Sigma X \backslash \{\0\};i^*\Ndot_{X}) &\cong \HH^{-1}(\bigcup_C C\backslash \{\0\};i^*\Ndot_{X}) \\
&\cong \bigoplus_C \HH^{-1}(C \backslash \{\0\};i^*\Ndot_{X}).
\end{align*}
This last term is easily seen to be (the sum of) global sections of the local system $H^{-1}(\Ndot_{X})_{|_{C \backslash \{\0\}}}$, which is just $\ker \{\Id-h_C\}$. Taking the alternating sums of the dimensions of the terms in the resulting short exact sequence 
\begin{equation}\label{eqn:semisimpleses}
0 \to H^{-1}(\Ndot_{X})_\0 \to\bigoplus_C \ker \{\Id -h_C\} \to \HH^0(\Sigma X,\Sigma X \backslash \{\0\};i^*\Ndot_{X}) \to 0
\end{equation}
yields the desired result.
\end{proof}

\begin{thm}\label{thm:semispacevspace}
There are isomorphisms
$$
\Gr_0^W \Q_X^\bullet[2] \cong \Gr_0^W \Ndot_X \cong W_0 \Ndot_X,
$$
and therefore the weight filtration on $\Q_X^\bullet[2]$ is concentrated in degrees $[0,2]$.
\end{thm}

\begin{proof}
Consider the short exact sequence (\ref{eqn:semisimpleses}) at the end of the proof of \thmref{thm:weightzeropart}; in particular, the middle term $\bigoplus_C \ker \{\Id-h_C\}$. 

The internal monodromy action of $H^{-1}(l^*\Ndot_{X})$ is semi-simple (See Remark 2.5 of \cite{qhomcriterion}), and each of the subspaces $\ker \{\Id-h_C\}$ is invariant under this action (and are therefore semi-simple as well). Consequently, 
$$
\HH^0(\Sigma X,\Sigma X \backslash \{\0\};i^*\Ndot_X) \cong H^0(W_0 \Ndot_{X})_\0
$$
is semi-simple as a $\Q$-vector space. 

\medskip

We again examine (\ref{eqn:semisimpleses}):
$$
0 \to H^{-1}(\Ndot_{X})_\0 \to\bigoplus_C \ker \{\Id -h_C\} \to H^0(W_0 \Ndot_{X})_\0 \to 0.
$$
The same argument used in \propref{prop:ndotvhs} shows that $H^{-1}(\Ndot_{X})_p$ carries a pure Hodge structure of weight zero. We claim that $\bigoplus_C \ker \{ \Id-h_C\}$ is also pure of weight zero. Indeed, for all components $C$, $\ker \{\Id-h_C\}$ is a direct summand of the stalk $H^{-1}(\Ndot_{X})_p$ for $p \in C \backslash \{\0\}$ generic, by semi-simplicity of the monodromy operator $h_C$, and is therefore also a weight zero Hodge structure. 

Since (\ref{eqn:semisimpleses}) is a short exact sequence of graded, polarizable mixed Hodge structures, it follows that $H^0(W_0 \Ndot_{X})_\0$ is also pure of weight zero, and we are done.

\end{proof}

\bigskip

\subsection{Topological Invariance of the Weight Filtration}\label{subsec:topinv}

We now show more generally that the weight filtration is an \textbf{invariant of the local topological-type} of parameterized surfaces $X$ in a neighborhood of a point $\0 \in X$, regardless of the space in which they are embedded. Before we state our theorem, however, we wish to recall some important facts.

First, we note that the perversity of the complex $\Q_X^\bullet[2]$ is topological: $\Q_X^\bullet[2]$ is perverse if and only if, for all $p \in X$, the real link $K_{X,p}$ is connected. This is easy to see by the cosupport condition, that is, the condition that, for all $k \in \Z$, the dimension of the space
$$
\cosupp^k \Q_X^\bullet[2] := \overline{ \{ p \in X \, | \, H_{\{p\}}^k(\Q_X^\bullet[2]) \neq 0 \}}
$$
is at most $k$-dimensional. The claim then follows from the fact that $H_{\{p\}}^k(\Q_X^\bullet[2]) \cong \wt H^{k+1}(K_{X,p};\Q)$ and applying the cosupport condition for $k=-1$. The support condition is trivially satisfied when $X$ is purely 2-dimensional.

Second, it is well known that $\Idot_X$ is a topological invariant of $X$, via a classical result of Goresky and MacPherson (see Section 4.3 of \cite{inthom2}). Consequently, if $\alpha : X \to Y$ is a homeomorphism of purely 2-dimensional reduced complex analytic spaces, and $\Q_X^\bullet[2]$ is perverse, then we know $\alpha^* \Idot_Y \cong \Idot_X$ and $\Q_X^\bullet[2] \cong \alpha^* \Q_Y^\bullet[2]$ is perverse; moreover, we obtain a morphism of short exact sequence of perverse sheaves on $X$:
\begin{equation*}
\xymatrix{
0 \ar[r]& \Ndot_X \ar[r]\ar@{-->}[d]&\Q_X^\bullet[2] \ar[d]^{\cong}\ar[r]&\Idot_X \ar[d]^{\cong}\ar[r]&0\\
0 \ar[r]& \alpha^*\Ndot_Y \ar[r] &\alpha^*\Q_Y^\bullet[2] \ar[r]&\alpha^*\Idot_Y \ar[r]&0
}
\end{equation*}
and it follows that we have an isomorphism $\Ndot_X \cong \alpha^*\Ndot_Y$. This shows that $W_i \Q_X^\bullet[2] \cong \alpha^*W_i \Q_Y^\bullet[2]$ for $i=1,2$. (Local) topological invariance of the weight zero term is more delicate, and will require further a homeomorphism of pairs $(X,\0_X) \xrightarrow{\alpha} (Y,\0_Y)$.

Recall \exref{exm:surfacecase}, where we show that, for parameterized surfaces $V(f) \subseteq \C^3$, the subset $\U \subseteq \Sigma f$ over which the normalization is a covering map is always equal to $\Sigma f \backslash \{\0\}$. This is, however, not the best viewpoint from which to analyze purely topological properties of $V(f)$. For example, the surface $V(y^2-x^3)$ in $\C^3$ is parameterized (it is a cross-product of a cusp and a line), and homeomorphic to $\C^2$, which is smooth. The problem is that the normalization of $V(y^2-x^3)$ is a \emph{bijection}, and thus $\Ndot_{V(y^2-x^3)}=0$. This is a non-issue for this particular example, since then $\Q_{V(y^2-x^3)}^\bullet[2] \cong \Idot_{V(y^2-x^3)}$, which is a topological invariant. We usually assume that $\supp \Ndot_X = \Sigma X$ to avoid these pathological cases, but \textbf{for completeness of the discussion below, we drop this assumption.}

To remedy this, we instead focus on $\U$ as a subset of $D_X:= \supp \Ndot_X$; under any homeomorphism $\alpha : (X,\0_X) \to (Y,\0_Y)$ of parameterized surfaces, $(D_X,\0_X)$ must be mapped homeomorphically onto $(D_Y,\0_Y)$. This is because $D_X$ is always purely $1$-dimensional at $\0$ (it is the support of a perverse sheaf with stalk cohomology concentrated in degree $-1$), and $\alpha(\0_X) = \0_Y$, so that the irreducible components of $D_X$ at $\0_X$ map bijectively onto the irreducible components of $D_Y$ at $\0_Y$. Thus, the ``correct" invariant definition of $\U$ is the disjoint union of all $C \backslash \{\0_X\}$ where $C$ is an irreducible component of $D_X$ at $\0_X$. 

More precisely, we show the following. 

\begin{thm}\label{thm:topinv}
If $(X,\0_X)$ and $(Y,\0_Y)$ are two parameterized surfaces and $\alpha$ is any homeomorphism of pairs $(X,\0_X) \to (Y,\0_Y)$, there are isomorphisms of perverse sheaves on $X$
$$
W_i \Q_X^\bullet[2] \cong \alpha^* W_i \Q_Y^\bullet[2],
$$
for $i=0,1,2$. 
\end{thm}

\begin{proof}


Following the above discussion, it remains to prove $W_0 \Q_X^\bullet[2] \cong \alpha^*\Q_Y^\bullet[2]$. To see this, we examine again the short exact sequence (\ref{eqn:gradedses}) obtained in \thmref{thm:graded}:
$$
0 \to W_0 \Q_X^\bullet[2] \to \Ndot_X \to i_*\Idot_{D_X}(\hat{l}^*\Ndot_{X}) \to 0.
$$
We only need to show $\Idot_{D_X}({\Ndot_{X}}_{|_{\U}}) \cong \alpha^*\Idot_{D_Y}({\Ndot_{Y}}_{|_{\alpha(\U)}})$, and the statement for $i=1$ together with the short exact sequence (\ref{eqn:gradedses}) will prove the $i=0$ case (recall again that $\hat{l}: \U \hookrightarrow X$, and $\U$ is the smooth, Zariski open dense subset of $D_X$ over which the normalization $\pi$ is a covering map).

\medskip

From the beginning of this subsection, since $\alpha$ maps $\U \subseteq D_X$ homeomorphically onto $\alpha(\U) = D_Y\backslash \{\0_Y\}$, the claim for $i=1$ gives isomorphisms 
\begin{equation}\label{eqn:smoothMHM}
{\Ndot_{X}}_{|_{\U}} \cong (\alpha^*\Ndot_Y)_{|_{\U}} \cong \alpha^*({\Ndot_Y}_{|_{\alpha(\U)}}).
\end{equation}
By \propref{prop:ndotvhs} we know the right hand side also underlies a polarizable variation of Hodge structure of weight 0, so that 
$$
\Gr_1^W \Q_X^\bullet[2] \cong \Idot_{D_X}({\Ndot_{X}}_{|_{\U}}) \cong \Idot_{D_X}(\alpha^*({\Ndot_{Y}}_{|_{\alpha(\U)}}))
$$
as polarizable Hodge modules of weight $1$ with strict support $D_X$. We finally wish to demonstrate the isomorphism
\begin{equation}\label{eqn:idot}
\Idot_{D_X}(\alpha^*({\Ndot_{Y}}_{|_{\alpha(\U)}})) \cong \alpha^*\Idot_{D_Y}({\Ndot_{Y}}_{|_{\alpha(\U)}}).
\end{equation}
Clearly, by (\ref{eqn:smoothMHM}), both objects restrict to the same underlying local system on $\U$. We are done if we can show that $\alpha^*\Idot_{D_Y}({\Ndot_{Y}}_{|_{\alpha(\U)}})$ has no perverse sub or quotient objects with support contained in $\{\0_X\}$ (these conditions uniquely characterize the intermediate extension up to isomorphism, see e.g. Section 4 of \cite{inthom2}), i.e., 
$$
{}^p H^0(m^!\alpha^*\Idot_{D_Y}({\Ndot_{Y}}_{|_{\alpha(\U)}})) = {}^p H^0(m^*\alpha^*\Idot_{D_Y}({\Ndot_{Y}}_{|_{\alpha(\U)}})) = 0.
$$
Since $\Idot_{D_Y}({\Ndot_{Y}}_{|_{\alpha(\U)}})$ is the intermediate extension of ${\Ndot_{Y}}_{|_{\alpha(\U)}}$ to all of $D_Y = \alpha(D_X)$, and $\alpha$ sends $\0_X$ to $\0_Y$, we trivially have 
$$
{}^p H^0(m_X^*\alpha^*\Idot_{D_Y}({\Ndot_{Y}}_{|_{\alpha(\U)}})) = 0.
$$
Since $m_X : \{\0_X\} \hookrightarrow D_X$ is the inclusion of a point, we know that ${}^p H^0(m_X^!)$ can be replaced with $H^0(m_X^!)$, and thus  
\begin{align*}
{}^p H^0(m_X^!\alpha^*\Idot_{D_Y}({\Ndot_{Y}}_{|_{\alpha(\U)}})) &= H^0(m_X^!\alpha^*\Idot_{D_Y}({\Ndot_{Y}}_{|_{\alpha(\U)}})) \\
&= \hyp^0(D_X,D_X \backslash \{\0_X\};\alpha^*\Idot_{D_Y}({\Ndot_{Y}}_{|_{\alpha(\U)}})) \\
&= \hyp^0(D_Y,D_Y \backslash \{\0_Y\}; \Idot_{D_Y}({\Ndot_{Y}}_{|_{\alpha(\U)}}))
\end{align*}
where the last isomorphism follows from the fact that $\alpha$ maps the pair $(D_X,D_X \backslash \{\0_X\})$ homeomorphically onto $(D_Y, D_Y \backslash \{\0_Y\})$, and taking the long exact sequence in relative hypercohomology with coefficients in $\alpha^*\Idot_{D_Y}({\Ndot_{Y}}_{|_{\alpha(\U)}}))$. As $\Idot_{D_Y}({\Ndot_{Y}}_{|_{\alpha(\U)}})$ has no perverse subobjects with support in $\{\0_Y\}$, we conclude 
$$
\hyp^0(D_Y, D_Y \backslash \{\0_Y\}; \Idot_{D_Y}({\Ndot_{Y}}_{|_{\alpha(\U)}}))=0,
$$
which establishes the desired isomorphism (\ref{eqn:idot}). Consequently, we again have a morphism of short exact sequences 
\begin{equation*}
\xymatrix{
0 \ar[r]& W_0 \Q_{X}^\bullet[2] \ar[r]\ar@{-->}[d]&\Ndot_{X} \ar[d]^{\cong}\ar[r]&i_*\Idot_{D_X}({\Ndot_{X}}_{|_{\U}})  \ar[d]^{\cong}\ar[r]&0\\
0 \ar[r]& \alpha^*W_0\Q_Y^\bullet[2] \ar[r] &\alpha^*\Ndot_Y \ar[r]&\alpha^*i_*\Idot_{D_Y}({\Ndot_{Y}}_{|_{\alpha(\U)}}) \ar[r]&0
}
\end{equation*}
from which we conclude $W_0 \Q_X^\bullet[2] \cong \alpha^*W_0 \Q_Y^\bullet[2]$, and we are done.
\end{proof}

\medskip

\begin{rem}
It is an interesting consequence of \thmref{thm:topinv} that, by \thmref{thm:qhomcriterion}, being a ``parameterized surface" is a purely topological property of the surface (instead of depending on the embedding, or the local ambient topological-type of the surface). 
\end{rem}

If $X$ is a reduced, purely 2-dimensional complex analytic space (not necessarily parameterized), we still know that the short exact sequence
$$
0 \to \Ndot_X \to \Q_X^\bullet[2] \to \Idot_X \to 0
$$
is a topological invariant of $X$. 

\begin{rem}\label{rem:steenbrink}
The classical result of Steenbrink and Stevens on the topological invariance of the weight filtration on the cohomology groups in the compact algebraic setting (Theorem 2.4 of \cite{STEENBRINK198463}) bears a resemblance to our results \thmref{thm:graded} and \thmref{thm:topinv} regarding the weight filtration on the constant sheaf the local analytic setting. In particular, the choice of compact subvariety $\Sigma \subseteq X$ with smooth Zariski-open complement $X \backslash \Sigma$, and the vanishing condition $W_{k-2}H^k(X) = 0$ for all $k$.
\end{rem}
\bigskip

\subsection{Examples and Finitely-Determined Maps}

\begin{exm}\label{exm:whitneyumbrella}
Let $f(x,y,z) = y^2-x^3-zx^2$, so that $V(f)$ is the Whitney umbrella. Then, $\Sigma f = V(x,y)$, and $V(f)$ has (smooth) normalization given by $\pi(u,t) = (u^2-t,u(u^2-t),t)$. Then, it is easy to see that the internal monodromy operator $h_C$ along the component $V(x,y)$ is multiplication by $-1$, so $\ker \{ \Id-h_C\} = 0$. Hence, 
$$
 \Gr_0^W \Q_{V(f)}^\bullet[2] = 0.
$$ 
\end{exm}

\begin{exm}\label{exm:notwhitney}
Let $g(x,y,z) = y^2-x^3-z^2x^2$, so that $\Sigma g = V(x,y)$. Then, $V(g)$ has smooth normalization given by $\pi(u,t)= (u^2-t^2,u(u^2-t^2),t)$, and $\pi^{-1}(\Sigma g) = V(u^2-t^2)$, with internal monodromy operator $h_C$ given by the identity map. Hence,
$$
\Gr_0^W \Q_{V(f)}^\bullet[2] = \Q_{\{\0\}}^\bullet.
$$
\end{exm}

\medskip

\begin{rem}
The above two examples are interesting as they are both (images of) one-parameter unfoldings with isolated instability of the same plane curve singularity, $V(y^2-x^3)$, i.e., a cusp. 

\end{rem}

\medskip

\begin{exm}\label{exm:weirdone}
 Let $f(x, y, z) = xz^2 - y^3$, so that $\Sigma f = V(y,z)$. Then, the normalization $\wt {V(f)}$ is equal to 
$$
\wt {V(f)} = V(u^2-xy,uy-xz,uz-y^2) \subseteq \C^4,
$$
(i.e., the affine cone over the twisted cubic) and the normalization map $\pi$ is induced by the projection $(u,x,y,z) \mapsto (x,y,z)$. By Section 4, \cite{qhomcriterion}, $\wt X$ is a rational homology manifold. The internal monodromy operator $h_C$ on $H^{-1}(\Ndot_{V(f)})_{|_{V(y,z) \backslash \{\0\}}}$ is trivial, so $\ker \{\Id-h_C\} \cong \Q$. Thus, 
$$
\Gr_0^W \Q_{V(f)}^\bullet[2] \cong \Q_{\{\0\}}^\bullet.
$$
\end{exm}

\medskip

\begin{exm}\label{exm:triplepoint}
$f(x,y,z) = xyz$, so $\Sigma f = V(x,y) \cup V(y,z) \cup V(x,z)$. Then, $|\pi^{-1}(\0)| = 3$, and the internal monodromy operators $h_C$ are all the identity. It then follows that 
$$
\Gr_0^W \Q_{V(f)}^\bullet[2] \cong \Q_{\{\0\}}^\bullet.
$$
\end{exm}

\medskip

\begin{exm}\label{exm:images}

Let $\pi : (\C^2,\0) \to (\C^3,\0)$ be a finitely-determined map, and set $\im \pi = V(f)$ for some reduced complex analytic function $f$ on $\C^3$ with $\dim_\0 \Sigma f = 1$ (for a precise definition of this, see \cite{mather},\cite{GAFFNEY1993185}, or \cite{PMIHES_1968__35__127_0}), i.e.., $\pi$ is finite and generically one-to-one (so that $\pi$ is the normalization of $V(f)$), and the generic transverse singularity type of $V(f)$ is that of a Morse function. We also note that $|\pi^{-1}(\0)|=1$ by assumption. Examples of this are \exref{exm:whitneyumbrella}, \exref{exm:notwhitney}, and \exref{exm:triplepoint} (although in this last example we are allowing more irreducible components at the origin).

Consequently, the stalk of the local system $\hat{l}^*\Ndot_{V(f)}$ at any point along any irreducible component of the critical locus $\Sigma f$ of $V(f)$ is $\Q$. This implies further that each of the summands $\ker \{ \Id - h_C\}$ in $W_0 \Q_{V(f)}^\bullet[2]$ must either be $0$ or all of $\Q$, according to whether or not the local system monodromy of $\hat{l}^*\Ndot_{V(f)}$ along $C$ is trivial. 

When $\pi$ is corank one (i.e., $\pi$ is a one-parameter unfolding of a finitely determined map from $(\C,\0)$ to $(\C^2,\0)$), we have seen from \exref{exm:whitneyumbrella} and \exref{exm:notwhitney} it is possible for for two different unfoldings of the same plane curve to have different weight filtrations. 


\end{exm}

\bigskip

When $W_0 \Q_{V(f)}^\bullet[2] = W_1 \Q_{V(f)}^\bullet[2] = 0$, i.e., $\Q_{V(f)}^\bullet[2] \cong \Idot_{V(f)}$, we know that $V(f)$ is a \textbf{rational homology manifold}. It would be very interesting to understand what the vanishing of $W_0 \Q_{V(f)}^\bullet[2]$ implies about the topology of $(V(f),\0)$, now that we know this is a topological invariant of $V(f)$ near $\0$. It is trivial to see that, if $V(f)$ is an equisingular deformation of a plane curve singularity, then $W_0 \Q_{V(f)}^\bullet[2] = 0$ (since this implies $\phi_L[-1]\Q_{V(f)}^\bullet[2] = 0$ for generic linear forms $L$ defined near $\0$).

\subsection{Unipotent Vanishing Cycles for Parameterized Surfaces}\label{subsec:surfacesvancycle}

The results of \subsecref{subsec:weightzero} and \secref{sec:vancycle} (i.e., the isomorphism $\Ndot_{V(f)} \cong \ker N (1)$) can now be rephrased as:
\begin{align*}
\Gr_1^W \Q_{V(f)}^\bullet[2] \cong \Gr_1^W \Ndot_{V(f)} &\cong \Gr_3^W \ker N \cong \Idot_{\Sigma f}(\hat{l}^*\Ndot_{V(f)}(-1)) \\
\Gr_0^W \Q_{V(f)}^\bullet[2] \cong \Gr_0^W \Ndot_{V(f)} &\cong \Gr_2^W \ker N \cong V_{\{\0\}}^\bullet (-1) 
\end{align*}
and $\Gr_k ^W \Ndot_{V(f)} = 0$ for $k < 0$, where $N = \frac{1}{2\pi i} \log T_u$ is the logarithm of the unipotent monodromy on $\phi_{f,1}[-1]\Q_{\U}^\bullet[3]$, and $(1)$ is the Tate twist operator.

\medskip

Since $\Q_\W^\bullet[3]$ is a pure Hodge module of weight $3$ (where $\W$ is some open neighborhood of the origin of $\0$ in $\C^3$ on which $f$ is defined), the vanishing cycles $\phi_{f,1}[-1]\Q_\W^\bullet[3]$ is a mixed Hodge module whose weight filtration is the \textbf{monodromy weight filtration} shifted by $3$. We then get the following result for free.

\begin{cor}
For parameterized surfaces $V(f)$, the graded pieces of the weight filtration on $\phi_{f,1}[-1]\Q_{\W}^\bullet[3]$ are as follows:
\begin{align*}
\Gr_{3+k}^W \phi_{f,1}[-1]\Q_\W^\bullet[3] \cong 
\begin{cases}
\Idot_{\Sigma f}(\hat{l}^*\Ndot_{V(f)}(-1)), & \text{ if $k = 0$,} \\
V_{\{\0\}}^\bullet (-1) & \text{ if $k=-1,1$,} \\
0, & \text{ otherwise,}
\end{cases}
\end{align*}
where $V$ is a $\Q$-vector space of dimension 
\begin{align*}
\dim_\Q V = 1-|\pi^{-1}(\0)| + \sum_C \dim \ker\{\Id -h_C\},
\end{align*}
with $\{C\}$ denoting the collection of irreducible components of $\Sigma f$ at $\0$, and for each component $C$, $h_C$ is the (internal) monodromy operator on the local system $H^{-1}(\Ndot_{V(f)})_{|_{C\backslash \{\0\}}}$.  
\end{cor}
 
There are no other pieces of the weight filtration of $\phi_{f,1}[-1]\Q_\W^\bullet[3]$ left unaccounted for; the levels below weight $3$ all lie in $\ker N$, and the ``Hard Lefschetz"-type isomorphism from (\ref{eqn:hardlefschetz}) give the levels above weight $3$. Since the weight filtration of $\ker N$ is concentrated in weights $2$ and $3$, the weight filtration of $\phi_{f,1}[-1]\Q_\U^\bullet[3]$ is concentrated in degrees $2,3,$ and $4$. 

Finally, we also note that there is no extra information contained in the \textbf{primitive graded pieces} or \textbf{Lefschetz decomposition} of the monodromy weight filtration on $\phi_{f,1}[-1]\Q_\W^\bullet[3]$, since 
\begin{align*}
N(\Gr_3^W \phi_{f,1}[-1]\Q_\W^\bullet[3]) &= 0 \text{, and }\\
N(\Gr_4^W \phi_{f,1}[-1]\Q_\W^\bullet[3]) &\cong \Gr_2^W \phi_{f,1}[-1]\Q_\W^\bullet[3].
\end{align*}
In light of \thmref{thm:topinv}, the weight filtration on the unipotent vanishing cycles is an invariant of the local topological-type of $V(f)$ near $\0$. We compare this with classical invariants like the Milnor number, which are invariants of the local \emph{ambient} topological-type of $V(f)$ near $\0$ (i.e., they depend on embedding of $V(f)$ near $\0$ in a smooth ambient space).

\section{Future Directions}\label{sec:futuredirections}

\begin{ques}\label{ques:ndothodge}
The most natural future direction to pursue is that of understanding the \textbf{Hodge filtration} on $\Q_{V(f)}^\bullet[2]$, so that, with \thmref{thm:weightzeropart} and \thmref{thm:semispacevspace}, we would have a complete understanding of $\Q_{V(f)}^\bullet[2]$ as a mixed Hodge module for parameterized surfaces in $\C^3$. 

\smallskip

The simplest class of examples on which to examine this Hodge filtration for parameterized surfaces is that of surfaces $V(f)$ whose transversal type along $\Sigma f$ is that of a quasi-homogeneous function with isolated singularities (this is the case for parameterized \textbf{affine toric surfaces}, and surfaces that are the image of \textbf{finitely determined maps} from $\C^2$ to $\C^3$). In such cases, the Milnor monodromy operator is semi-simple, and thus $\ker N = \phi_f[-1]\Q_\W^\bullet[n+1]$ by a result of Joseph Steenbrink (Theorem 1, \cite{steenbrinkquasihom}); moreover one has explicit generators of the graded pieces $\Gr_p^F \Gr_{p+q}^W H^n(F_{f,p};\C)$, where $F_{f,p}$ denotes the \textbf{Milnor fiber} of $f$ at $p$. We would thus be able to completely determine the ``Hodge theory" of such surfaces.
\end{ques}

\bigskip

\begin{ques}\label{ques:ndotbeta}
One notes that the formula for the dimension of the vector space $V = H^0(\Gr_0^W \Q_{V(f)}^\bullet[2])_\0$ in \thmref{thm:weightzeropart} is very similar to the \textbf{beta invariant}, $\beta_f$, of a hypersurface $V(f)$ with one-dimensional singular locus (defined by David Massey in \cite{betainv}, and further explored by the author and Massey in \cite{specialcases}). 

Does its vanishing have a similar geometric significance to the vanishing of $\beta_f$? What, if any, is the geometric significance of the dimension of $H^0(\Gr_0^W\Q_{V(f)}^\bullet[2])_\0$?

\smallskip

It is possible for $V = 0$; this happens, e.g., for the Whitney umbrella $V(y^2-x^3-zx^2)$ for which $\Sigma f$ is smooth at the origin, but this is not a sufficient condition.  Indeed, the critical locus of $V(xz^2-y^3)$ is also smooth at $\0$, but $V = \Q$ (see \exref{exm:whitneyumbrella} and \exref{exm:weirdone}).

\smallskip

However, we may distinguish these examples by noting that, for generic linear forms $L$, the normalization map $\pi : \widetilde {V(f)} \to V(f)$ is a \textbf{simultaneous normalization} of the the family $\pi_\xi : \widetilde {V(f)} \cap (L \circ \pi)^{-1}(\xi) \to V(f,L-\xi)$ for all $\xi \in \C$ small in the case of the Whitney umbrella, but \textbf{not} for the surface $V(xz^2-y^3)$. Is this true in general? This would make the perverse sheaf $\Gr_0^W \Q_{V(f)}^\bullet[2]$ very relevant to \textbf{L\^{e}'s Conjecture} regarding the equisingularity of parameterized surfaces in $\C^3$:

\medskip

\begin{conj}[L\^{e}, \cite{bobleconj},\cite{1983Ntes}]
Suppose $(V(f),\0) \subseteq (\C^3,\0)$ is a reduced hypersurface with $\dim_\0 \Sigma f = 1$, for which the normalization of $V(f)$ is a bijection. Then, in fact, $V(f)$ is the total space of an equisingular deformation of plane curve singularities.
\end{conj}

\medskip

Thus, L\^{e}'s Conjecture is concerned with parameterized surfaces for which $\Ndot_{V(f)} = 0$, i.e., where $\Q_{V(f)}^\bullet[2] \cong \Idot_{V(f)}$ is a semi-simple perverse sheaf (or, more generally, a pure Hodge module of weight $2$). In particular, this would obvious imply $\Gr_0^W \Q_{V(f)}^\bullet[2] = 0$.
\end{ques}

 
 
 

\printbibliography

\end{document}